\documentclass{amsart}
\usepackage{amsmath, amssymb, mathrsfs, verbatim, multirow}
\usepackage[shortlabels]{enumitem} 
\usepackage[all]{xy}

\newtheorem{Teo}{Theorem}[section]
\newtheorem{Prop}[Teo]{Proposition}

\newtheorem{Claim}[Teo]{Claim}

\theoremstyle{definition}
\newtheorem{Def}[Teo]{Definition}

\newtheorem{Obs}[Teo]{Remark}

\newtheorem{Exa}[Teo]{Example}

\newcommand{\Q}{\mathbb{Q}}

\newcommand{\Z}{\mathbb{Z}}
\newcommand{\N}{\mathbb{N}}

\newcommand{\Llr}{\Longleftrightarrow}
\newcommand{\lra}{\longrightarrow}
\newcommand{\Lra}{\Longrightarrow}
\newcommand{\VR}{\mathcal{O}}

\newcommand{\MI}{\mathfrak{m}}

\begin{document}
\title[Structure of graded algebra]{On the structure of the graded algebra associated to a valuation}
\author{M. S. Barnab\'e}
\author{J. Novacoski}
\author{M. Spivakovsky}
\thanks{During part of the realization of this project Novacoski was supported by a research grant from Funda\c c\~ao de Amparo \`a Pesquisa do Estado de S\~ao Paulo (process number 2017/17835-9).}

\begin{abstract}
The main goal of this paper is to study the structure of the graded algebra associated to a valuation. More specifically, we prove that the associated graded algebra ${\rm gr}_v(R)$ of a subring $(R,\MI)$ of a valuation ring $\mathcal{O}_v$, for which $Kv:=\VR_v/\MI_v=R/\MI$, is isomorphic to $Kv\left [t^{v(R)}\right ]$, where the multiplication is given by a \textit{twisting}. We show that this twisted multiplication can be chosen to be the usual one in the cases where the value group is free or the residue field is closed by radicals. We also present an example that shows that the isomorphism (with the trivial twisting) does not have to exist.
\end{abstract}

\maketitle

\section{Introduction}

Let $\mathcal{O}_v$ be the valuation ring of a valuation $v$ on a field $K$. Let $Kv=\mathcal{O}_v/\mathfrak{m}_v$ be the residue field and $\Gamma_v$ the value group of $v$. For a ring $R \subseteq \mathcal{O}_v$ let $v(R)$ denote the image of $R \setminus\{0\}$ in $\Gamma_v$, $\mathfrak{m}:=R\cap \mathfrak{m}_v$ and consider the graded algebra
\begin{equation*}
{\rm gr}_v(R)= \displaystyle  \bigoplus_{\gamma \in v(R)} \mathcal{P}_{\gamma}/\mathcal{P}^{+}_{\gamma}, 
\end{equation*}
where 
\begin{equation*}
\mathcal{P}_{\gamma}=\{x \in R \,|\, v(x) \geq \gamma\} \mbox{ and } \mathcal{P}^{+}_{\gamma}=\{x \in R \,|\, v(x) > \gamma\}.
\end{equation*}

Graded algebras play an important role to understand extensions of valuations (see for instance, \cite{msro} and \cite{Vaq}). Graded algebras are also central objects in the approach of Teissier for local uniformization (see \cite{Tei} and \cite{Tei2}). Key polynomials are defined to be elements whose natural images in the graded algebra are irreducible and have some additional properties.

The semigroup ring $Kv\left [t^{v(R)}\right ]$ is the set of finite formal sums
\[
\sum_{i=1}^na_i\cdot t^{\gamma_i},\mbox{ with }a_i\in Kv\mbox{ and }\gamma_i\in v(R), 1\leq i\leq n,
\]
where
\[
\sum_{i=1}^na_i\cdot t^{\gamma_i}+\sum_{i=1}^nb_i\cdot t^{\gamma_i}=\sum_{i=1}^n(a_i+b_i)\cdot t^{\gamma_i}
\]
and
\[
\left(\sum_{i=1}^na_i\cdot t^{\gamma_i}\right)\times\left(\sum_{j=1}^mb_j\cdot t^{\gamma_j}\right):=\sum_{i,j}a_ib_j\cdot t^{\gamma_i+\gamma_j}.
\]
In \cite{Cut}, it is stated that ${\rm gr}_v(R)$ is isomorphic to $Kv\left [t^{v(R)}\right ]$. This is not in general true, as our Example \ref{Examatheus} shows. However, the main result of this paper is that these objects are isomorphic if we consider a \textit{twisted multiplication} on $Kv\left [t^{v(R)}\right ]$. We will call a \textbf{choice function on $v(R)$} any right inverse of $v$, i.e., a map $\epsilon: v(R)\lra R$ such that $v(\epsilon (\gamma))=\gamma$ for every $\gamma\in v(R)$. We will always consider choice functions for which $\epsilon(0)=1$. For each choice function, we can define a map (which we will call a \textbf{twisting})
\[
\overline\epsilon: v(R)\times v(R)\lra Kv,\mbox{ by }\overline{\epsilon}(\gamma,\gamma'):=\left(\frac{\epsilon(\gamma)\epsilon(\gamma')}{\epsilon(\gamma+\gamma')}\right)v.
\]
Then we can define a map
\[
\times_\epsilon:Kv\left [t^{v(R)}\right ]\times Kv\left [t^{v(R)}\right ]\lra Kv\left [t^{v(R)}\right ], (a,b)\mapsto a\times_\epsilon b,
\]
by setting $t^\gamma\times_\epsilon t^{\gamma'}:=\overline\epsilon(\gamma,\gamma')\cdot t^{\gamma+\gamma'}$ and extending it to $Kv\left [t^{v(R)}\right ]$ in the obvious way. We prove that this map satisfies the multiplication axioms, making $Kv\left [t^{v(R)}\right]$ into a commutative ring. This multiplication will be called the \textbf{twisted multiplication} induced by $\epsilon$ and we will denote this ring by $Kv\left [t^{v(R)}\right ]_{\epsilon}$.

Our first main theorem is the following.
\begin{Teo} \label{main1} 
If $Kv=R/\MI$, then for every choice function $\epsilon$ we have that
\[
{\rm gr}_{v}(R)\simeq Kv\left [t^{v(R)}\right ]_{\epsilon}.
\]
\end{Teo}

Then we study whether there exists a choice function $\epsilon$ for which the twisted multiplication is the usual one, i.e., such that $t^\gamma\times_\epsilon t^{\gamma'}=t^{\gamma+\gamma'}$ for every $\gamma,\gamma'\in v(R)$. Observe that this is the same as saying that $\overline\epsilon\equiv 1v$. We show the following.
\begin{Teo}\label{main2}
If $\Gamma_v$ is a free group or $Kv$ is closed by radicals, then there exists a choice function $\epsilon$ for which $\overline\epsilon\equiv 1v$.
\end{Teo}

The condition of $\Gamma_v$ being a free group is satisfied for a large class of valuations. For instance, it is satisfied when $v$ is Abhyankar, i.e., when the equality holds in the inequality
\begin{equation*}
{\rm rat.rk}(v) + {\rm tr.deg}_{R/\mathfrak{m}} Kv \leq \dim R
\end{equation*}
(see Abhyankar's Inequality \cite[Appendix 5]{Zar} for more details).

We observe that this problem is similar to the one addressed by Kaplansky in \cite{Kap}, where the main goal is to understand maximal fields. There it is studied whether the extension of a valued field to a maximal extension is unique. The concept of \textit{power set} appearing there is very similar to the twistings appearing here.

\textbf{Acknowledgements.} We would like to thank the anonymous referee for a careful reading and for pointing out a few mistakes in an earlier version of this paper.

\section{The isomorphism}\label{CFAI}
We start by presenting some easy remarks about a given choice function $\epsilon$.
\begin{Obs}
\textbf{(i)} The map $\overline{\epsilon}$ is symmetric and $\overline{\epsilon}(0,\alpha)=1v$, for every $\alpha\in v(R)$.\\
\textbf{(ii)} Given three elements $\alpha, \beta, \gamma \in v(R)$, we have
\begin{equation}\label{multprod}
\begin{array}{rl}
\overline{\epsilon}(\alpha,\beta) \cdot \overline{\epsilon}(\alpha+\beta,\gamma) & = \displaystyle \left(\frac{\epsilon(\alpha)\epsilon(\beta)}{\epsilon(\alpha+\beta)}\right)v \cdot \left(\frac{\epsilon(\alpha+\beta)\epsilon(\gamma)}{\epsilon(\alpha+\beta+\gamma)}\right)v \\[15pt]
& = \displaystyle \left(\frac{\epsilon(\alpha)\epsilon(\beta)\epsilon(\gamma)}{\epsilon(\alpha+\beta+\gamma)}\right)v \\[15pt]
& =  \displaystyle \left(\frac{\epsilon(\alpha)\epsilon(\beta + \gamma)}{\epsilon(\alpha+\beta+\gamma)}\right)v \cdot  \left(\frac{\epsilon(\beta)\epsilon(\gamma)}{\epsilon(\beta+\gamma)}\right)v  \\[15pt]
& =\overline{\epsilon}(\alpha,\beta+\gamma) \cdot \overline{\epsilon}(\beta,\gamma).
\end{array}
\end{equation}
\end{Obs}

We now show that the map $\times_\epsilon$ induces a structure of commutative ring on $Kv\left [t^{v(R)}\right ]$.

\begin{Prop} \label{prop1}
For each choice function $\epsilon$, the set $Kv\left [t^{v(R)}\right ]$ with the standard sum and the multiplication given by $\times_{\epsilon}$ is a commutative ring.
\end{Prop}

\begin{proof}
Under addition, $Kv\left [t^{v(R)}\right ]$ is an abelian additive group by definition. For the multiplication, from the way that it is defined, it is enough to check the properties only for elements of the form $f=f_{\alpha}t^{\alpha}$. Take $f=f_{\alpha}\cdot t^{\alpha}, g=g_{\beta}\cdot t^{\beta}\in Kv\left [t^{v(R)}\right ]$. 
Then
\[
f \times_{\epsilon}g  = (f_{\alpha}g_{\beta})\cdot t^{\alpha}\times_{\epsilon}t^{\beta} =  (g_{\beta}f_{\alpha})\cdot t^{\beta}\times_{\epsilon}t^{\alpha} = g \times_{\epsilon}f.
\]
Hence $\times_\epsilon$ is commutative. Also, for $1:=(1v)\cdot t^0 \in Kv\left [t^{v(R)}\right ]$ we have
\[
1\times_{\epsilon}f =f_{\alpha}\overline{\epsilon}(0,\alpha)\cdot t^{0+\alpha}=f_\alpha \cdot t^\alpha=f,
\]
so $Kv\left [t^{v(R)}\right ]$ has a unit element.

To prove that $\times_\epsilon$ is distributive, we take $h=h_\beta\cdot t^\beta$ and observe that
\begin{equation*}
\begin{array}{rcl}
f\times_{\epsilon}(g+h)&=& f_{\alpha}\cdot t^{\alpha}\times_{\epsilon}(g_{\beta}+h_{\beta})\cdot t^{\beta} =f_{\alpha}(g_{\beta}+h_{\beta})\cdot t^{\alpha}\times_{\epsilon}t^{\beta}\\[10pt]
&=& (f_{\alpha}g_{\beta})\cdot t^{\alpha}\times_{\epsilon}t^{\beta}+(f_{\alpha}h_{\beta})\cdot t^{\alpha}\times_{\epsilon}t^{\beta} \\[10pt]
&=& f\times_{\epsilon}g + f\times_{\epsilon}h.
\end{array}
\end{equation*}

Lastly, in order to prove that $\times_{\epsilon}$ is associative it is enough to prove it for elements of the form $f=t^\alpha, g=t^\beta$ and $h= t^\gamma$. In this case, we have
\begin{equation*}
\begin{array}{rcl}
(f\times_{\epsilon}g)\times_{\epsilon}h   & = &\left(\overline{\epsilon}(\alpha,\beta)\cdot t^{\alpha + \beta}\right) \times_{\epsilon} t^{\gamma}=\overline{\epsilon}(\alpha,\beta)\overline{\epsilon}(\alpha+\beta,\gamma)\cdot  t^{\alpha + \beta + \gamma}\\[10pt]
			 &\stackrel{(\ref{multprod})}{=} &\overline{\epsilon}(\alpha,\beta+\gamma)\overline{\epsilon}(\beta,\gamma)\cdot t^{\alpha + \beta + \gamma}= \overline{\epsilon}(\beta,\gamma)\cdot t^{\alpha} \times_{\epsilon}  t^{\beta + \gamma}\\[10pt]
			 & = &f\times_{\epsilon}(g\times_{\epsilon}h).
\end{array}
\end{equation*}
This concludes the proof.
\end{proof}

\begin{Def}
We will denote by $Kv\left [t^{v(R)}\right ]_{\epsilon}$ the ring $(Kv\left [t^{v(R)}\right ], +, \times_{\epsilon})$ and by $Kv\left [t^{v(R)}\right ]$ the ring $(Kv\left [t^{v(R)}\right ], +, \times)$, where $\times$ denotes the usual product.
\end{Def}

For any element $x \in R$ with $v(x) = \gamma$, the natural image of $x$ in $\mathcal{P}_{\gamma}/\mathcal{P}^{+}_{\gamma}$ is a homogeneous element of ${\rm gr}_v(R)$ of degree $\gamma$, which we will denote by ${\rm in}_v(x)$.  Note that, for $x, y \in R$, we have ${\rm in}_v(x)={\rm in}_v(y)$ if and only if $v(x - y) > v(x)$.

For $x \in R\setminus\{0\}$ we define
\begin{equation*}
\psi({\rm in}_{v}(x))=\left(\frac{x}{D(x)}\right)v\cdot t^{v(x)},
\end{equation*}
where $D=\epsilon \circ v$. Extend $\psi$ to ${\rm gr}_v(R)$ by setting
\begin{equation*}
\displaystyle \psi\left(\sum_{\alpha \in \Delta}{\rm in}_{v}(x_{\alpha})\right)=\sum_{\alpha \in \Delta}\psi({\rm in}_{v}(x_{\alpha})) \mbox{ and } \psi(0)=0,
\end{equation*}
where $\Delta$ is a finite subset of $v(R)$ and $v(x_{\alpha})=\alpha.$

\begin{Obs} Note that
\[
t^{v(x)}\times_{\epsilon}t^{v(y)} = \left( \frac{D(x)D(y)}{D(xy)} \right)v \cdot t^{v(xy)}
\]
for all $x,y \in K$.
\end{Obs}

Theorem \ref{main1} will follow from the following proposition.
\begin{Prop} \label{prop2} 
The map $\psi$ is an injective ring homomorphism from ${\rm gr}_{v}(R)$ to $Kv\left [t^{v(R)}\right ]_{\epsilon}.$ Moreover, if $(R,\mathfrak{m})$ is a local ring dominated by $\mathcal{O}_v$ such that $R/\mathfrak{m}=Kv,$ then this map is an isomorphism.
\end{Prop}

\begin{proof}

We will prove first that $\psi$ is injective and well-defined, i.e., that
\begin{equation}\label{welldefined}
{\rm in}_v(x)={\rm in}_v(y) \Leftrightarrow  \psi({\rm in}_v(x))=\psi({\rm in}_v(y)),
\end{equation}
for every $x,y \in R$. For $x,y \in R$ we have ${\rm in}_v(x)={\rm in}_v(y)$ if and only if
\begin{equation*}
v(x)=v(y) \mbox{ and } v(x-y)>v(x).
\end{equation*}
Since $v(x)=v(D(x)),$ this is equivalent to 
\begin{equation*}
v(x)=v(y) \mbox{ and } v\left(\frac{x-y}{D(x)}\right)>0,
\end{equation*}
which is equivalent to
\begin{equation*}
v(x)=v(y) \mbox{ and } \left(\frac{x-y}{D(x)}\right) v=0.
\end{equation*}
This happens if and only if $\psi({\rm in}_v(x))=\psi({\rm in}_v(y))$.

To prove that $\psi$ is an additive group homomorphism it is enough to show that if $v(x)=v(y)$, then
\[
\psi({\rm in}_v(x)+{\rm in}_v(y))=\psi({\rm in}_v(x))+\psi({\rm in}_v(y)).
\]
To verify this property we have to consider two possibilities: $v(x+y)>v(x)$ or $v(x+y)=v(x)=v(y).$ In the first case we have
\begin{equation*}
\psi({\rm in}_v(x)+{\rm in}_v(y))=0=\left(\frac{x+y}{D(x)}\right)v=\psi({\rm in}_v(x))+\psi({\rm in}_v(y)). 
\end{equation*}
In the second case,
\begin{equation*}
\begin{array}{rcl}
\psi({\rm in}_v(x)+{\rm in}_v(y))& =& \displaystyle \psi({\rm in}_v(x+y)) = \left(\frac{x+y}{D(x+y)}\right)v \cdot t^{v(x+y)} \\[15pt]
& =&\displaystyle \left(\frac{x}{D(x)}\right)v \cdot t^{v(x)} + \left(\frac{y}{D(y)}\right)v \cdot t^{v(y)} \\[10pt]
& =&\displaystyle \psi({\rm in}_v(x))+\psi({\rm in}_v(y)).
\end{array}
\end{equation*}

To show that $\psi$ is multiplicative, it is enough to show that for $x,y \in R$ we have
\[
\psi({\rm in}_v(x){\rm in}_v(y))=\psi({\rm in}_v(x)) \times_{\epsilon} \psi({\rm in}_v(y)).
\]
Since 
\[
t^{v(xy)}=t^{v(x)+v(y)}=\left(\frac{D(xy)}{D(x)D(y)}\right)v\cdot t^{v(x)}\times_\epsilon t^{v(y)}
\]
we have
\begin{equation*}
\begin{array}{rcl}
\psi({\rm in}_v(x){\rm in}_v(y)) & = & \displaystyle \psi({\rm in}_v(xy)) = \left(\frac{xy}{D(xy)}\right)v \cdot t^{v(xy)} \\[20pt]
												& = &\displaystyle \left(\frac{xy}{D(xy)} \frac{D(xy)}{D(x)D(y)}\right)v \cdot t^{v(x)} \times_{\epsilon}  t^{v(y)} \\[20pt]
												& = &\displaystyle \left(\frac{x}{D(x)}\right)v \cdot t^{v(x)} \times_{\epsilon}  \left(\frac{y}{D(y)}\right)v \cdot t^{v(y)} \\[20pt]
											& = &\displaystyle \psi({\rm in}_v(x)) \times_{\epsilon} \psi({\rm in}_v(y)).
\end{array}
\end{equation*}
This proves that $\psi$ is an injective ring homomorphism.

Assume now that $R/\MI=Kv$ and take any element $x\cdot t^{\alpha}\in Kv\left [t^{v(R)}\right ]$. Choose any $z' \in R$ such that $z'v=x$. It is easy to verify that for $z=\epsilon(\alpha)z'$ we have
\begin{equation*}
\displaystyle \psi\left({\rm in}_vz\right)=x\cdot t^{\alpha}.
\end{equation*}
Therefore $\psi$ is surjective.
\end{proof}
\begin{Obs}
We observe that the map $\psi$ is \textit{graded} ring homomorphism, i.e., for every $x\in R$ we have $\deg({\rm in}_v(x))=\deg(\psi({\rm in}_v(x)))$. Here, $\deg$ is the map that associates each homogeneous element to the corresponding element in the grading semigroup.
\end{Obs}
We will now study when we can choose $\epsilon$ such that the multiplication $\times_\epsilon$ is the usual one, i.e., when the map $\overline\epsilon$ is constant and equal to $1$. We start with the following easy remark.

\begin{Obs}\label{remabetweegrandon}
The map $\overline\epsilon$ is constant and equal to $1$ if and only if ${\rm in}_v\circ \epsilon$ is a semigroup homomorphism between $v(R)$ and the multiplicative semigroup of ${\rm gr}_v(R)$. Indeed, for $\alpha,\beta\in v(R)$ we have
\begin{displaymath}
\begin{array}{rcl}
\overline\epsilon(\alpha,\beta)=1v &\Llr & \displaystyle\left(\frac{\epsilon(\alpha)\epsilon(\beta)}{\epsilon(\alpha+\beta)}\right)v=1v\Llr v\left(\frac{\epsilon(\alpha)\epsilon(\beta)}{\epsilon(\alpha+\beta)}-1\right)>0\\[15pt]
& \Llr & \displaystyle v\left(\epsilon(\alpha)\epsilon(\beta)-\epsilon(\alpha+\beta)\right)>v(\epsilon(\alpha+\beta))\\[10pt]
&\Llr & {\rm in}_v\circ \epsilon(\alpha+\beta) = \left({\rm in}_v\circ \epsilon(\alpha)\right)\cdot\left({\rm in}_v\circ \epsilon(\beta)\right)
\end{array}
\end{displaymath}
\end{Obs}

We will prove the two parts of Theorem 1.2 --- the case when $\Gamma_v$ is free and the case when $K_v$ is closed by radicals --- as two separate theorems.

\begin{Teo} \label{thmparte2} 
Assume that $\Gamma_v$ is a free group. Then there exists a choice function $\epsilon:\Gamma_v \to K$ such that $\overline{\epsilon}\equiv 1v$.
\end{Teo}

\begin{proof}
Write  $\Gamma_v= \displaystyle \bigoplus_{i \in I} \gamma_i\mathbb{Z}$ for some $\gamma_i\in \Gamma_v$. For each $\gamma_i$ we choose $z_i \in K$ such that $v(z_i)=\gamma_i.$ For each $\alpha\in \Gamma_v$, we write
\begin{equation}\label{unique}
\alpha=n_1\gamma_{i_1}+ \cdots + n_r\gamma_{i_r} \in \Gamma_v,\mbox{ for some }n_i \in \mathbb{Z}.
\end{equation}
Define
\begin{equation*}
\epsilon(\alpha):=z_{i_1}^{n_1} \cdots z_{i_r}^{n_r}
\end{equation*}
Since the expression in \eqref{unique} is unique, this map is well-defined. Moreover, $\epsilon$ is a choice function because $v(\epsilon(\alpha))=\alpha$, for every $\alpha\in v(R)$.

Take $\alpha,\beta\in\Gamma_v$ and write (adding $n_i$'s or $m_j$'s equal to zero, if necessary)
\[
\alpha=n_1\gamma_{i_1}+ \cdots + n_r\gamma_{i_r} \mbox{ and } \beta=m_1\gamma_{i_1}+ \cdots + m_r\gamma_{i_r}.
\]
Then
\begin{equation*}
\epsilon(\alpha)\epsilon(\beta)=z_{i_1}^{n_1} \cdots z_{i_r}^{n_r} \cdot z_{i_1}^{m_1} \cdots z_{i_r}^{m_r}= z_{i_1}^{n_1+m_1} \cdots z_{i_r}^{n_r+m_r}=\epsilon(\alpha+\beta).
\end{equation*}
Therefore,
\[
\overline \epsilon(\alpha,\beta)=\left(\frac{\epsilon(\alpha)\epsilon(\beta)}{\epsilon(\alpha+\beta)}\right)v=1v,
\]
which is what we wanted to prove.
\end{proof}

\begin{Obs}
Observe that for every $(R,\MI)$ as in Theorem \ref{main2}, the map constructed above would induce a choice function on $v(R)$ that satisfies the required properties.
\end{Obs}

We will now prove the second part of Theorem \ref{main2}.

\begin{Teo}\label{cbr}
If $Kv$ is closed by radicals, then there exists a choice function $\epsilon:\Gamma_v \to K$ such that $\overline{\epsilon}\equiv 1v$.
\end{Teo}

\begin{proof} Consider the set
\[
\mathcal P:= \left\{(\Phi, \epsilon_{\Phi})\mid \Phi \subseteq \Gamma_v\mbox{ is a subgroup and } \epsilon_{\Phi}:\Phi\lra K\mbox{ is a choice function with } \overline{\epsilon}_{\Phi} \equiv 1v\right\}.
\]
We define a partial order on $\mathcal P$ by setting
\begin{equation*}
(\Phi_1,\epsilon_{{\Phi}_1})\prec (\Phi_2,\epsilon_{{\Phi}_2}) \Llr {\Phi}_1 \subseteq {\Phi}_2 \mbox{ and } \epsilon_{{\Phi}_2} \mbox{ extends } \epsilon_{{\Phi}_1}.
\end{equation*}

For any fixed nonzero $\alpha \in \Gamma_v$ let
\[
\langle \alpha\rangle:=\{n\alpha\mid n\in\Z\}
\]
be the subgroup of $\Gamma_v$ generated $\alpha$. We pick $z\in K$ such that $v(z)=\alpha$ and define
\[
\epsilon_{\langle \alpha\rangle}:\langle \alpha\rangle\lra K\mbox{ by }\epsilon_{\langle\alpha\rangle}(n \alpha)=z^n.
\]
It is easy to check that $\left(\langle \alpha\rangle,\epsilon_{\langle \alpha\rangle}\right)\in \mathcal P$, so $\mathcal P$ is not empty.

Let $Q=\{(\Phi_i,\epsilon_{{\Phi}_i})\}_{i \in I}$ be a totally ordered subset of $\mathcal P$. Set $\Phi_Q:= \displaystyle \bigcup_{i \in I}\Phi_i$. For each $\alpha\in \Phi$ we choose $i\in I$ such that $\alpha\in \Phi_i$ and set $\epsilon_{\Phi_Q}(\alpha)=\epsilon_{\Phi_i}(\alpha)$. Since $Q$ is totally ordered, this map is well-defined. It is straightforward to check that $(\Phi_Q,\epsilon_{\Phi_Q})\in \mathcal P$ and that $(\Phi_Q,\epsilon_{\Phi_Q})$ is an upper bound for $Q$. Hence, we can apply Zorn's Lemma to obtain a maximal element $(\Phi,\epsilon_{\Phi})$ of $\mathcal P$.

We will show that $\Phi=\Gamma_v$ and this will conclude our proof. Suppose, aiming for contradiction, that $\Phi \neq \Gamma_v$ and take $\gamma \in \Gamma_v \setminus \Phi.$ We have two possibilities:
\[
\langle\gamma \rangle\cap \Phi=\emptyset\mbox{ or }\langle\gamma \rangle\cap \Phi\neq\emptyset.
\]

In the first case we take $x_{\gamma} \in K$ with $v(x_{\gamma})=\gamma$ and define $\Psi:=\Phi+\langle\gamma\rangle$ and $\epsilon_{\Psi}: \Psi \to K$ by
\begin{equation*}
\epsilon_{\Psi}(\alpha+n\gamma):=\epsilon_\Phi(\alpha)\cdot x_{\gamma}^n,\mbox{ for }\alpha \in \Phi\mbox{  and }n \in \mathbb{Z}.
\end{equation*}
Since $n\gamma \notin \Phi$ for every $n\in\Z$, we have that for $\alpha, \beta \in \Phi$ and $n, m \in \mathbb{Z}$
\begin{equation*}
\alpha + n\gamma = \beta +m\gamma \Lra \alpha=\beta \mbox{ and } n=m.
\end{equation*}
Therefore, $\epsilon_{\Psi}$ is well-defined and one can check that $(\Psi,\epsilon_\Psi)\in\mathcal P$. 

Now, if $\langle\gamma \rangle\cap \Phi\neq\emptyset$, we define
\[
n_0 = \min \{n \geq 2\mid n \gamma \in \Phi\}.
\]
Let $x_0=\epsilon_{\Phi}(n_0 \gamma) \in K$ and choose $x_{\gamma} \in K$ such that $v(x_{\gamma})=\gamma$. Since $Kv$ is closed by radicals, there exists $a \in K$ such that 
\begin{equation}\label{va=0}
\displaystyle \left( \frac{x_0}{x_{\gamma}^{n_0}}\right)v=(a^{n_0})v.
\end{equation}

Every element in $\Psi:=\Phi+\langle\gamma\rangle$ is of the form $\alpha+n\gamma$, for a uniquely determined $\alpha\in\Phi$ and $n$, $0\leq n<n_0$. Indeed, if $\alpha+n\gamma=\alpha'+m\gamma$, $0\leq m\leq n<n_0$, then $(n-m)\gamma=\alpha'-\alpha\in\Phi$. Then, $n=m$ and $\alpha=\alpha'$ and this gives us the uniqueness. For the existence, every element in $\Psi$ is of the form $\alpha'+r\gamma$ for some $r\in\Z$. Using Euclidean division, we write $r=n_0l+n$ with $0\leq n<n_0$ and $l\in\Z$. Then $\alpha:=\alpha'+n_0\gamma\in \Phi$ which is what we needed.

Define $\epsilon_{\Psi}: \Psi \lra K$ by
\begin{equation*}
\epsilon_{\Psi}(\alpha + n \gamma) = \epsilon_{\Phi}(\alpha) \cdot (ax_{\gamma})^n,
\end{equation*}
where $\alpha + n\gamma \in \Psi$, with $\alpha\in\Phi$ and $0\leq n\leq n_0$. Then $\epsilon_{\Psi}$ is well-defined.

We will show that $(\Psi,\epsilon_{\Psi})\in\mathcal P$ and since $(\Phi,\epsilon_\Phi)\prec (\Psi,\epsilon_\Psi)$ we obtain a contradiction with the maximality of $(\Phi,\epsilon_\Phi)$ in $\mathcal P$. Let $\alpha + n_1\gamma$ and $\beta + n_2 \gamma$ be two elements in $\Psi$, where $0 \leq n_1, n_2 <n_0$. We will consider two cases.

\noindent\textbf{Case 1:} If $n_1 + n_2 < n_0$, then	
\begin{equation*}
\begin{array}{rl}
\overline{\epsilon}_{\Psi}(\alpha + n_1\gamma, \beta + n_2 \gamma) & = \displaystyle \left(\frac{\epsilon_{\Psi}(\alpha+ n_1 \gamma) \,\, \epsilon_{\Psi}(\beta+n_2\gamma)}{\epsilon_{\Psi}(\alpha+\beta+(n_1+n_2)\gamma)}\right)v \\[15pt]
& = \displaystyle \left(\frac{\epsilon_{\Phi}(\alpha) \,\, (ax_{\gamma})^{n_1} \,\, \epsilon_{\Phi}(\beta) \,\, (ax_{\gamma})^{n_2}}{\epsilon_{\Phi}(\alpha+\beta) \,\, (ax_{\gamma})^{n_1+n_2}}\right)v \\[15pt]
& = \displaystyle \left(\frac{\epsilon_{\Phi}(\alpha)\,\, \epsilon_{\Phi}(\beta) }{\epsilon_{\Phi}(\alpha+\beta)} \right)v \cdot \left(\frac{(ax_{\gamma})^{n_1} \,\, (ax_{\gamma})^{n_2} }{ (ax_{\gamma})^{n_1+n_2} }\right)v \\[15pt]
& = \overline{\epsilon}_{\Phi}(\alpha, \beta)\cdot 1v  =1v
\end{array}
\end{equation*}

\noindent\textbf{Case 2:} If $n_1 + n_2 \geq n_0$, then $n_0 \leq n_1 + n_2 <2n_0$ and
\[
(\alpha + n_1\gamma)+(\beta + n_2 \gamma)=(\alpha + \beta + n_0 \gamma) + (n_1 + n_2 - n_0)\gamma.
\]
Since $0\leq n_1+n_2-n_0<n_0$, we have
\[
\epsilon_\Psi((\alpha+n_1\gamma)+(\beta+n_2\gamma))=\epsilon_\Phi(\alpha+\beta+n_0\gamma)\cdot (ax_\gamma)^{n_1+n_2-n_0}.
\]
Then
\begin{equation*}
\begin{array}{rcl}
\overline{\epsilon}_{\Psi}(\alpha + n_1\gamma, \beta + n_2 \gamma) & =&\displaystyle \left(\frac{\epsilon_{\Psi}(\alpha+ n_1 \gamma) \,\, \epsilon_{\Psi}(\beta+n_2\gamma)}{\epsilon_{\Psi}((\alpha+n_1\gamma)+(\beta+n_2\gamma))}\right)v \\[15pt]
& = &\displaystyle \left(\frac{\epsilon_{\Phi}(\alpha) \,\, (ax_{\gamma})^{n_1} \,\, \epsilon_{\Phi}(\beta) \,\, (ax_{\gamma})^{n_2}}{\epsilon_{\Phi}(\alpha+\beta+n_0\gamma) \,\, (ax_{\gamma})^{n_1+n_2-n_0}}\right)v \\[15pt]
& = &\displaystyle \left(\frac{\epsilon_{\Phi}(\alpha) \,\, \epsilon_{\Phi}(\beta) \,\, (ax_{\gamma})^{n_0}}{\epsilon_{\Phi}(\alpha+\beta+n_0\gamma) }\right)v\\
\end{array}
\end{equation*}
Since $\displaystyle\left(\frac{\epsilon_{\Phi}(n_0\gamma)}{(ax_{\gamma})^{n_0}}\right)v=1v$, we multiply on the right side of the above equation by it to obtain
\begin{equation*}
\begin{array}{rcl}
\overline{\epsilon}_{\Psi}(\alpha + n_1\gamma, \beta + n_2 \gamma) & = &\displaystyle \left(\frac{\epsilon_{\Phi}(\alpha)\epsilon_{\Phi}(\beta)(ax_{\gamma})^{n_0}}{\epsilon_{\Phi}(\alpha+\beta+n_0\gamma) }\right)v \cdot  \left(\frac{\epsilon_{\Phi}(n_0\gamma)}{(ax_{\gamma})^{n_0}}\right)v \\[15pt]
& =&\displaystyle\left(\frac{\epsilon_{\Phi}(\alpha)\epsilon_{\Phi}(\beta)\epsilon_{\Phi}(n_0\gamma)}{\epsilon_{\Phi}(\alpha+\beta+n_0\gamma) }\right)v=  1v.	
\end{array}
\end{equation*}
This completes the proof.
\end{proof}		

We now present an example where we have to use a non-trivial twisting to obtain the isomorphism in our main theorem.

\begin{Exa}\label{Examatheus}
Consider ${\bf P} \subseteq \N$ the set of all prime numbers and ${\bf Q}=\{x_p\mid p\in\textbf{P}\}$ a set of independent variables. Consider
\[
K:=\Q(x_p\mid x_p \in {\bf Q})\mbox{ and }\Gamma:=\sum_{p \in {\bf P}}\frac{1}{p}\Z \subset \mathbb{Q}.
\]
We will construct a valuation $v:K^\times\lra \Gamma$. Start by defining $v(x_p):=1/p$ for every $p\in\textbf{P}$ and $v(a)=0$ for every $a\in\Q$. This defines $v$ on every monomial $a\textbf{Q}^{\lambda}\in K$, because
\[
v\left(a\textbf{Q}^{\lambda}\right)=v\left(a\prod_{\lambda(p)\neq 0}x_p^{\lambda(p)}\right)=\sum_{\lambda(p)\neq 0}v(a)+\lambda(p)v(x_p)=\sum_{\lambda(p)\neq 0}\frac{\lambda(p)}{p}.
\]
We extend $v$ to $K$ monomialy, i.e., by setting
\begin{equation*}
v\left(\sum_{i=1}^na_i\textbf{Q}^{\lambda_i}\right)=\min_{a_i\neq 0}\left\{v\left(a_i\textbf{Q}^{\lambda_i}\right)\right\},
\end{equation*}	
in $\Q[x_p\,|\, x_p \in {\bf Q}]$ and $v(P/Q)=v(P)-v(Q)$ in $K$.
\begin{Def}
For any polynomial
\[
P=\sum_{i=1}^na_i\textbf{Q}^{\lambda_i}\in\Q[x_p\,|\, x_p \in {\bf Q}]
\]
we define the \textbf{initial part of $P$ with respect to $v$} by
\[
{\rm ip}_v(P):=\sum_{v(a_i\textbf{Q}^{\lambda_i})=v(P)}a_i\textbf{Q}^{\lambda_i}.
\]
A polynomial $P$ is said to be initial if ${\rm ip}_v(P)=P$.
\end{Def}

\begin{Obs}\label{remakinch}
It is easy to verify that for $P,Q\in \Q[x_p\mid x_p \in {\bf Q}]$ we have
\begin{description}
\item[(i)] ${\rm in}_v\left({\rm ip}_v(P)\right)={\rm in}_v\left(P\right)$; and
\item[(ii)] ${\rm in}_v(P)={\rm in}_v(Q)\Llr {\rm ip}_v(P)={\rm ip}_v(Q).$ In particular, if $P$ and $Q$ are initial polynomials, then
\[
{\rm in}_v(P)={\rm in}_v(Q)\Llr P=Q.
\]

\end{description}
\end{Obs}

\begin{Def}
We will say that the choice function $\epsilon:\Gamma\lra K$ is initial, if for every $\alpha\in \Gamma$, $\epsilon(\alpha)=\frac{P}{Q}$ where both $P$ and $Q$ are initial.
\end{Def}
\begin{Claim}\label{claiminchois}
If $\epsilon$ is a choice function such that $\overline\epsilon\equiv 1v$, then there exists an initial choice function $\epsilon'$ such that $\overline{\epsilon'}\equiv 1v$.
\end{Claim}
\begin{proof}
For each $\alpha\in \Gamma$ we write $\epsilon(\alpha)=\frac{P_\alpha}{Q_\alpha}\in K$. Set $\epsilon'(\alpha):=\frac{{\rm ip}_v(P_\alpha)}{{\rm ip}_v(Q_\alpha)}$. Then
\[
\alpha=v(\epsilon(\alpha))=v\left(\frac{P_\alpha}{Q_\alpha}\right)=v(P_\alpha)-v(Q_\alpha)=v({\rm ip}_v(P_\alpha))-v({\rm ip}_v(Q_\alpha))=v(\epsilon'(\alpha)),
\]
hence $\epsilon'$ is a right inverse for $v$. Moreover, by Remark \ref{remabetweegrandon},
\begin{displaymath}
\begin{array}{rcl}
\overline{\epsilon}(\alpha,\beta)=1v&\Llr&{\rm in}_v\left(\epsilon(\alpha)\right)\cdot{\rm in}_v\left(\epsilon(\beta)\right)={\rm in}_v\left(\epsilon(\alpha+\beta)\right)\\[10pt]
&\Llr&\displaystyle{\rm in}_v\left(\frac{P_\alpha}{Q_\alpha}\right)\cdot{\rm in}_v\left(\frac{P_\beta}{Q_\beta}\right)={\rm in}_v\left(\frac{P_{\alpha+\beta}}{Q_{\alpha+\beta}}\right)\\[15pt]
&\Llr&\displaystyle{\rm in}_v\left(\frac{{\rm ip}_v(P_\alpha)}{{\rm ip}_v(Q_\alpha)}\right)\cdot{\rm in}_v\left(\frac{{\rm ip}_v(P_\beta)}{{\rm ip}_v(Q_\beta)}\right)={\rm in}_v\left(\frac{{\rm ip}_v(P_{\alpha+\beta})}{{\rm ip}_v(Q_{\alpha+\beta})}\right)\\[10pt]
&\Llr&\overline{\epsilon'}(\alpha,\beta)=1v.
\end{array}
\end{displaymath}
\end{proof}
In view of Claim \ref{claiminchois}, we can assume that $\epsilon$ is an initial choice function.
\begin{Claim}\label{claimnohf}
For an initial choice function $\epsilon:\Gamma\lra K$, if $\overline\epsilon\equiv 1v$, then $\epsilon(n\alpha)=\epsilon(\alpha)^n$ for every $n\in\N$ and $\alpha\in\Gamma$.
\end{Claim}
\begin{proof}
Take arbitrary $\alpha \in \Gamma$ and $n\in\N$. Set 
\begin{equation*}
\epsilon(\alpha)=\frac{P}{Q} \quad \mbox{ and } \quad \epsilon(n\alpha)=\frac{P'}{Q'}.
\end{equation*}
Our assumption guarantees that $\displaystyle\left(\frac{\epsilon(\alpha)^n}{\epsilon(n\alpha)}\right)v=1v$. This gives us
\begin{equation*}
0<\displaystyle v \left(\frac{\epsilon(\alpha)^n}{\epsilon(n\alpha)} -1 \right)=v\left(\frac{P^nQ'}{Q^nP'} -1 \right)=v\left(\frac{P^nQ'-Q^nP'}{Q^nP'}\right).
\end{equation*}
Hence
\[
v\left(P^nQ'-Q^nP'\right)>v\left(Q^nP'\right).
\]
By Remark \ref{remakinch} \textbf{(ii)}, we have $P^nQ'=Q^nP'$ and consequently
\[
\epsilon(\alpha)^n=\frac{P^n}{Q^n}=\frac{P'}{Q'}=\epsilon(n\alpha).
\]
\end{proof}

As a consequence of Claim \ref{claimnohf}, we obtain
\begin{equation}\label{eqroot}
\epsilon(1) = \epsilon \left( p \, \frac{1}{p} \right) = \epsilon \left(\frac{1}{p}\right)^p \mbox{ for every } p \in {\bf P}.
\end{equation}
For an element $R=P/Q\in K$ we set $\deg(R):=\max\{\deg_{\textbf{Q}}(P),\deg_{\textbf{Q}}(Q)\}$. Then $\deg(R^n)=n\deg(R)$ for every $R\in K$ and $n\in\N$. From \eqref{eqroot}, we have
\[
p\mid\deg(\epsilon(1))\mbox{ for every }p\in \textbf{P}.
\]
This implies that $\deg(\epsilon(1))=0$ and consequently $\epsilon(1)\in \Q$. Hence $v(\epsilon(1))=0$, which is a contradiction to our assumption that $\epsilon$ is a right inverse for $v$.
\end{Exa}

\noindent{\footnotesize MATHEUS DOS SANTOS BARNAB\'E\\
Departamento de Matem\'atica--UFSCar\\
Rodovia Washington Lu\'is, 235\\
13565-905 - S\~ao Carlos - SP\\
Email: {\tt matheusbarnabe@dm.ufscar.br} \\\\

\noindent{\footnotesize JOSNEI NOVACOSKI\\
Departamento de Matem\'atica--UFSCar\\
Rodovia Washington Lu\'is, 235\\
13565-905 - S\~ao Carlos - SP\\
Email: {\tt josnei@dm.ufscar.br} \\\\

\noindent{\footnotesize MARK SPIVAKOVSKY\\
CNRS UMR 5219 and Institut de Math\'ematiques de Toulouse\\
118, rte de Narbonne, 31062 Toulouse cedex 9, France\\
and\\ 
Instituto de Matem\'aticas (Unidad Cuernavaca) LaSol, UMI CNRS 2001\\
Universidad Nacional Aut\'onoma de M\'exico\\
Av. Universidad s/n. Col. Lomas de Chamilpa\\
C\'odigo Postal 62210, Cuernavaca, Morelos, M\'exico.\\
Email: {\tt spivakovsky@math.univ-toulouse.fr}
\end{document}